\documentclass[11pt, reqno]{amsart}
\usepackage[utf8]{inputenc}
\usepackage{amsmath,bm}
\usepackage{amsthm}
\usepackage{amssymb}
\usepackage{mdframed}
\usepackage{amsfonts}
\usepackage{caption}
\usepackage{MnSymbol}
\usepackage{circuitikz, booktabs}
\usepackage{tcolorbox}
\usepackage{blindtext,amsmath,amsthm,amsfonts, textcomp, mathrsfs, mathtools}
\usepackage{relsize}
\usepackage[shortlabels]{enumitem}

\usepackage[top=27mm,bottom=27mm,left=27mm,right=27mm]{geometry}
\newtheorem{theorem}{Theorem}[section]
\newtheorem{conjecture}[theorem]{Conjecture}
\newtheorem{corollary}[theorem]{Corollary}

\newtheorem*{theorem*}{Theorem}
\newtheorem*{remark*}{Remark}
\newtheorem*{problem*}{Problem}
\newtheorem*{conjecture*}{Conjecture}
\newtheorem*{question*}{Question}

\newtheorem{proposition}[theorem]{Proposition}

\newcommand{\rom}[1]{\uppercase\expandafter{\romannumeral #1\relax}}
\newcommand{\Q}{\mathbb{Q}}
\newcommand{\Z}{\mathbb{Z}}

\newcommand{\R}{\mathbb{R}}
\newcommand{\C}{\mathbb{C}}
\newcommand{\F}{\mathbb{F}}
\newcommand{\rF}{\mathcal{Q}}

\def\house#1{{%
    \setbox0=\hbox{$#1$}
    \vrule height \dimexpr\ht0+1.4pt width .4pt depth \dp0\relax
    \vrule height \dimexpr\ht0+1.4pt width \dimexpr\wd0+2pt depth \dimexpr-\ht0-1pt\relax
    \llap{$#1$\kern1pt}
    \vrule height \dimexpr\ht0+1.4pt width .4pt depth \dp0\relax
}}

\begin{document}

\title[A $p$-adic criterion for Lehmer's conjecture]{A $p$-adic criterion for Lehmer's conjecture}

\author[Anup B. Dixit]{Anup B. Dixit}
\author[Sushant Kala]{Sushant Kala}

\address{Department of Mathematics\\ Institute of Mathematical Sciences (HBNI)\\ CIT Campus, IV Cross Road\\ Chennai\\ India-600113}
\email{anupdixit@imsc.res.in}

\address{Department of Mathematics\\ Institute of Mathematical Sciences (HBNI)\\ CIT Campus, IV Cross Road\\ Chennai\\ India-600113}
\email{sushant@imsc.res.in}
\date{}

\begin{abstract}
For a non-zero algebraic number $\alpha$ of degree $d$, let $h(\alpha)$ denote its logarithmic Weil height. It is known  that when $h(\alpha)$ is small, and $d$ is large, the conjugates of $\alpha$ are clustered near the unit circle and have angular equidistribution in the complex plane about the origin. In this paper, we establish a $p$-adic analogue of this result by obtaining lower bounds for $h(\alpha)$ in terms of the number of its conjugates that lie in a finite extension of $\Q_p$, for some prime $p$. As a consequence, we prove Lehmer's conjecture for all $\alpha$ such that $\gg \sqrt{d\log d}$ many of its conjugates lie in a finite extension of $\Q_p$.
\end{abstract}



\thanks{Research of the first author was partially supported by the INSPIRE Faculty Fellowship.}

\maketitle

% \tableofcontents

\section{\bf Introduction}

\medskip

Let $h$ denote the absolute logarithmic Weil height on algebraic numbers. It is a fundamental tool in Diophantine geometry and arithmetic dynamics. It provides a means of quantifying the arithmetic complexity of algebraic numbers. Analytically, it connects to potential theory on Berkovich spaces, and geometrically, it encodes the distribution of algebraic points on varieties.\\

A classical result by Northcott \cite{Northcott} states that only finitely many algebraic numbers have bounded degree and bounded height. Complementing this, a theorem of Kronecker \cite{Kronecker} characterizes the algebraic numbers of height zero as precisely the roots of unity. The fundamental example $h(2^{1/d}) = (\log 2)/d$ illustrates that the height can become arbitrarily small as one ranges over algebraic numbers that are not roots of unity.\\ 

We fix $\overline{\Q}$ to be the algebraic closure of $\Q$ within $\C$, and henceforth, all algebraic extensions of $\Q$ are considered as subfields of $\overline{\Q}$. Let $\mu_{\infty}$ denote the set of all roots of unity in $\overline{\Q}$. A central open problem in the study of heights is Lehmer's conjecture \cite{lehmer}, which proposes a uniform lower bound on the height of non-torsion algebraic numbers.  

\begin{conjecture}[Lehmer]
    There exists an absolute constant $C>0$ such that
    \begin{equation}\label{Lehmer}
        h(\alpha) > \frac{C}{[\Q(\alpha):\Q]}
    \end{equation}
    for all non-zero $\alpha\in \overline{\Q}\setminus\mu_{\infty}$.
\end{conjecture} 

\noindent
Lehmer’s conjecture was originally formulated in terms of the Mahler measure of polynomials and is frequently referred to as \textit{Lehmer’s problem} in the literature. While the conjecture remains unsolved in general, it has been settled for various classes of algebraic numbers. For a comprehensive account of this problem, the reader is referred to the excellent survey articles \cite{Smyth} and \cite{Verger-Gaugry}. 

\medskip

The example discussed earlier, where $h(2^{1/d}) = (\log 2)/d$, shows that the bound in \eqref{Lehmer} is optimal. Nevertheless, when we restrict ourselves to special classes of algebraic numbers, it becomes possible to establish significantly stronger lower bounds on the height. For instance, consider the class of totally real algebraic numbers-those whose Galois conjugates all lie in $\R$ under any choice of embedding $\overline{\Q} \hookrightarrow \C$. For such numbers Schinzel \cite{Schinzel} proved in 1973 that the height is bounded below by an absolute constant independent of the degree, 
\begin{equation*}
    h(\alpha) \geq \frac{1}{2} \log \left(\frac{1+\sqrt{5}}{2}\right),
\end{equation*} 
for all $\alpha \neq 0, \pm 1$ that are totally real.\\

This bound reflects a deeper geometric rigidity: totally real numbers cannot have all its conjugates very close to $\pm 1$. More generally, one can establish lower bounds on the height of algebraic numbers, whose conjugates are restricted to a certain region of the complex plane. The following result due to Mignotte \cite{mignotte} quantitatively captures how the angular distribution of the conjugates relates to the height.

\begin{theorem*}[\cite{masser-book}, Theorem 15.2]
    For $\alpha \in \overline{\Q}$, let $d = [\Q(\alpha):\Q]$. For any $\theta$ with $0\leq \theta \leq 2\pi$ the number $n$ of conjugates of $\alpha$ in any fixed sector, based at the origin, of angle $\theta$ satisfies
    \begin{equation}\label{equidistribution}
        \left | n - \frac{\theta}{2\pi} d\right| \leq 24 \left( d^{2/3} \left(\log 2d\right)^{1/3} + d h(\alpha)^{1/3}\right).
     \end{equation}
\end{theorem*}
This result can be viewed as a quantitative refinement of a classical theorem due to Erd\H{o}s and Tur\'{a}n \cite{Erdos} on the distribution of roots of complex polynomials.  They showed  that if $f(z) \in \C[z]$ has small maximum modulus on the unit circle $|z| = 1$, its roots are concentrated near the unit circle and exhibit near-uniform angular distribution. 

\medskip
In 1997, Bilu \cite{Bilu} proved a beautiful result on the equidistribution of Galois conjugates in the complex plane for points with small height. Roughly speaking, if $\alpha$ is an algebraic number of height close to zero, then the Galois conjugates of $\alpha$ are equidistributed along the unit circle in the sense of weak convergence. Subsequently, a quantitative version of Bilu's theorem was established by Petsche \cite{Petsche}, and generalized by Favre and Rivera-Letelier \cite{Favre I} (see also the corrigendum \cite{Favre II}).

\medskip

The goal of this article is to realise the above phenomenon in the non-archimedean setting. Let $p$ be a rational prime. Denote by $\Q_p$ the field of $p$-adic numbers. We fix an algebraic closure $\overline{\Q_p}$ over $\Q_p$ and shall assume that  all the algebraic extensions of $\Q_p$ are subfields of $\overline{\Q_p}$. Further, for each prime $p$, we fix an embedding $i_p \, : \,  \overline{\Q} \hookrightarrow \overline{\Q_p}$. We say an algebraic number $\alpha \in \overline{\Q}$ \textit{lies in an extension $K$} of $\Q_p$ if $i_p(\alpha) \in K$. For ease of notation, we will identify $i_p(\alpha)$ and $\alpha$, when $\alpha\in K$. Viewing $\R$ as the completion of $\Q$ at the archimedean place, one can define a non-archimedean analogue of totally real numbers, called the totally $p$-adic numbers. An algebraic number $\alpha$ is said to be \textit{totally $p$-adic} if its image lies in $\Q_p$ for any embedding $\overline{\Q} \hookrightarrow \C_p$, where $\C_p$ denotes the completion of $\overline{\Q_p}$.

\medskip

A recent theorem of Pottmeyer \cite{pottmeyer} provides a $p$-adic analogue of Schinzel's result for totally real numbers. He showed that if $\alpha$ is totally $p$-adic and not a $(p-1)$-th root of unity,
\begin{equation*}
    h(\alpha) \geq \frac{\log (p/2)}{p+1}.
\end{equation*}
More generally, let $K/\Q_p$ be a finite extension with ramification index $e$ and inertia degree $f$. Let $S$ be the set of algebraic numbers all of whose Galois conjugates lie in $K$. Then, a result of Bombieri and Zannier \cite[Theorem 2]{BZ} implies that 
\begin{equation*}
    \liminf_{\alpha \in S} h(\alpha) \geq \frac{\log p}{2e (p^{f}+1)}.
\end{equation*}
\\

A generalized equidistribution theorem due to Baker and Rumely \cite[Theorem 2.3]{Baker-Rumely} provides a non-Archimedean analogue of Bilu’s equidistribution theorem, formulated using the framework of Berkovich spaces. It is important to note that the aforementioned result requires all the Galois conjugates to lie in a fixed local field. Therefore, it is natural to study the case when only a proportion of conjugates lie in a fixed local field of small degree. It follows from a recent work of the authors \cite{Dixit-Kala} that when a positive proportion of the Galois conjugates of $\alpha$ lie in a finite extension of $\Q_p$, the height $h(\alpha)$ is bounded below by an absolute constant, independent of the degree. More precisely, let $K/\Q_p$ be a finite extension with residue field $\F_q$ and suppose $\psi_q$ is the proportion of Galois conjugates of $\alpha$ in $K$. Then
\begin{equation*}
    h(\alpha) \gg \psi_q \, \frac{\log q}{q+1}.
\end{equation*}
It is important to note that the quantity $\psi_q$ is originally defined in \cite{Dixit-Kala} as the proportion of prime ideals with norm $q$ in $\Q(\alpha)$, but it is not difficult to see that it coincides with the proportion of conjugates of $\alpha$ lying in a local field $K$ with residue field $\F_q$.

\medskip

 From  \eqref{equidistribution}, we can conclude that algebraic numbers of large degree, whose conjugates are not angularly equidistributed in the complex plane, are forced to have large height. In other words, for all $\alpha\in\overline{\Q}$ satisfying
\begin{equation*}
    \left | n - \frac{\theta}{2\pi} d\right| \gg \left(d^2 \log d\right)^{1/3},
\end{equation*}
with the implied constant $> 24$, Lehmer's conjecture holds. The main goal of this paper is to establish a $p$-adic analogue of the above phenomenon. In particular, we prove Lehmer's conjecture for a class of algebraic numbers by quantifying how many conjugates must lie in a fixed local field. 

\medskip

We now set up the notation necessary to state our results. Let $K$ be a finite extension of $\Q_p$ with residue field $\F_q$. Let $\alpha \in \overline{\Q} \setminus \mu_\infty$ be a non-zero algebraic number of degree $d$, and $\alpha_1, \alpha_2, \ldots, \alpha_d$ be the Galois conjugates of $\alpha$ in $\overline{\Q}$. Define 
\begin{equation*}
\mathcal{S}_K(\alpha) = \left\{ \,\, \alpha_i\,\, | \,\,  \alpha_i \in  K,\,\, 1\leq i\leq d\right\}.
\end{equation*}

\begin{theorem}\label{Main Theorem}
    Let $K$ be as above. Then, for a non-zero $\alpha \in \overline{\Q}$
    \begin{equation*}
      h(\alpha) \geq \, \frac{1}{2 \, [K : \Q_p]} \, \left( \frac{| \mathcal{S}_K(\alpha) |}{d} \right)^2 \frac{\log q}{q+1} \,\, + \,\, o_d(1),
    \end{equation*}
    where $o_d(1)$ tends to $0$ as $d$ tends to infinity.
\end{theorem}

The method of proof of the above theorem can be used to derive the following consequences.

\begin{proposition}\label{Main Theorem 2}
    Let $K$ be as above. For $c>0$, consider the set
    \begin{equation*}
       S : = \left\{ \alpha \in  \overline{\Q}\setminus\mu_{\infty}  \,\, : \,\,| \mathcal{S}_K(\alpha) | \ \ \geq \ \ \sqrt{\left(\frac{2 (q+1) \, [K : \Q_p]}{\log q}d \right)\, \, \left( c + \frac{\log {(q d)}}{2} \right)  } \right\}.
    \end{equation*} 
    Then, $h(\alpha) \geq \frac{c}{d}$ for all $\alpha\in S$ and hence Lehmer's conjecture holds for $S$.
\end{proposition}

This proves that a condition much weaker than ``positive proportion" of conjugates in $K$ is sufficient to establish Lehmer's conjecture. In particular, it is enough to show that $\gg \sqrt{d\log d}$ many conjugates of $\alpha$ lie in a fixed local field, with a suitable implied constant. As a corollary of the above proposition, we have the following.

\begin{corollary}\label{corollary I}
    Let $K$ be as above. Suppose 
    \begin{equation*}
        S:= \left\{ \alpha \in \overline{\Q} \setminus \{ \mu_{\infty} \}\,\, : \,\, | \mathcal{S}_K(\alpha) | \geq \sqrt{q^3 [K:\Q_p]\, d \log d} \right\}.
    \end{equation*}
    Then, Lehmer's conjecture holds for $S$. In fact, $d\cdot h(\alpha)\to \infty$ as $d\to\infty$.
\end{corollary}

\noindent
An archimedean analog in the spirit of Theorem \ref{Main Theorem} with a condition on the number of real conjugates of $\alpha$ was obtained by Garza in \cite{garza} using Schinzel's method.

\medskip

We now generalize our setup to capture the behaviour of conjugates spread across multiple local fields over several primes. Define
$$
\rF_{p,\alpha} := \{ \, \Q_p(\alpha_i) \,\, | \,\, 1 \leq i \leq d \,\},
$$ 
as the set of \textit{distinct}  finite extensions of $\Q_p$ generated by the Galois conjugates of $\alpha$. For each $K \in \rF_{p, \alpha}$, denote by $e_K$ and $f_K$ the ramification index and inertia degree of $K$, respectively, $n_K=[K:\Q_p]$ and $\F_{q_K}$ the residue field. 

\begin{theorem}\label{Main Theorem 3}
Let $\alpha \in \overline{\Q} \setminus \mu_{\infty}$ be a non-zero algebraic integer. Then, 
\begin{equation*}
 h(\alpha) \geq \frac{1}{2} \, \, \mathlarger{\mathlarger{\sum}}_p \,\, \operatorname{max} \left\{ 0, \,\, \left( \frac{1}{\sum\limits_{K \in \mathcal{Q}_{p, \alpha}} e_K \, q_K}  - \frac{1}{d^2} \sum_{K\in \mathcal{Q}_{p, \alpha}}\,\, \frac{| \mathcal{S}_K(\alpha) |}{e_K} \right) \right\} \log p \,  - \frac{\log d}{2d}.
\end{equation*}
\end{theorem}

\noindent
\textbf{Remark.} For a totally $p-$adic algebraic integer, not a root of unity, Theorem \ref{Main Theorem 3} implies that
$$
h(\alpha) \geq  \frac{\log p}{2p} \operatorname{max} \left\{ 0, \,\, \left( 1 - \frac{p}{d} \right) \right\}- \frac{\log d}{2d}.
$$
Hence, by Northcott's theorem (Theorem \ref{Northcott}) 
$$
\liminf_{\alpha \in \mathcal{O}_{\Q^{tp}}}h(\alpha) \geq \frac{\log p}{2p},
$$
where $\Q^{tp}$ is the maximal totally $p$-adic field over $\Q$ and $\mathcal{O}_{\Q^{tp}}$ its ring of integers. Hence, we obtain a weaker version of Pottmeyer's result.

\medskip
It is important to note that our method can be used to extend Theorem \ref{Main Theorem 3} to algebraic numbers at the cost of a slightly weaker lower bound. Since Lehmer's conjecture is the theme of this article, we restrict ourselves to algebraic integers, as Lehmer's conjecture trivially holds for algebraic non-integers.

\medskip

As a consequence of the above theorem, we obtain the following corollary.

\begin{corollary}\label{corollary 1}
  Let $\alpha \in \overline{\Q} \setminus \mu_{\infty}$ be a non-zero algebraic integer. Then
   $$
  h(\alpha) \geq \frac{1}{2} \mathlarger{\mathlarger{\sum}}_p \operatorname{max} \left\{0, \,\, \frac{1}{d^2}\sum_{ \substack{ K \in \mathcal{Q}_{p, \alpha}}} \left( \frac{{| \mathcal{S}_K(\alpha) |}^2}{q_K} - | \mathcal{S}_K(\alpha) |  \right) \, \, \frac{\log q_K}{n_K} \right\}- \frac{\log d}{2d}.
   $$
\end{corollary}

In other words, if conjugates of $\alpha$ lie in several local fields $K$ with small inertia degree, then we obtain a meaningful lower bound for $h(\alpha)$. This can be compared with a theorem of Mignotte \cite{Mignotte 1}, which states that if there exists a prime $p \leq d \log d$, which splits completely in $\Q(\alpha)$, then Lehmer's conjecture holds for such $\alpha$. An improvement due to Silverman (alluded to in the literature but unpublished) states that if there are $d$ distinct prime ideals in $\Q(\alpha)$ with norm $\leq \sqrt{d \log d}$, then $h(\alpha)\geq c/d$ for some $c>0$. Indeed, if there are several primes with small norms in $F=\Q(\alpha)$, by Weil's explicit formula, one can obtain a sharper bound on the absolute discriminant $|\Delta(F/\Q)|$. For $\alpha$ an algebraic integer, we have $\Delta(F/\Q)$ divides $ D(m_{\alpha})$, the discriminant of the minimal polynomial of $\alpha$. Thus, Mahler's inequality (Theorem \ref{mahler}) can be applied and one obtains a lower bound on $h(\alpha)$. This phenomenon is explicitly demonstrated by the authors in \cite{Dixit-Kala-2}.

\bigskip

\section{\bf Preliminaries}\label{Pre}
\medskip

In this section, we recall some definitions and results which will be used towards the proof of our theorems. Let $\alpha \in \Q$ be a non-zero algebraic number of degree $d$ and $m_\alpha(x) = a_d x^d + \cdots + a_1 x + a_0 \in \Z[x]$ be the minimal polynomial of $\alpha$. Then, its Mahler measure is defined as
\begin{equation*}
    M(\alpha) := |a_n| \prod_{i} \max(1,|\alpha_i|),
\end{equation*}
where $\alpha_i$'s denote the Galois conjugates of $\alpha$. The logarithmic Weil height of $\alpha$, in terms of the Mahler measure, is thus defined as
\begin{equation*}\label{height}
    h(\alpha) = \frac{\log M(\alpha)}{d} \,.
\end{equation*}\\
The Weil height gives a partial ordering on algebraic numbers with bounded degree. This follows from a classical result of Northcott \cite{Northcott}.

\begin{theorem}[Northcott]\label{Northcott}
   For any fixed positive integers $d$ and 
$H$, there are only finitely many algebraic numbers $\alpha \in \overline{\Q}$ of degree at most $d$ and height at most $H$.
\end{theorem}

\bigskip

Recall that the discriminant of a polynomial $f(x) = b_n x^n + b_{n-1}x^{n-1} + \ldots + b_0 \in \Z[x]$ is an integer given by
$$
D(f) = {b_n}^{2n-2} \prod_{i > j} (\beta_i - \beta_j)^2,
$$
where $\beta_i$'s are all roots of $f$ in $\C$. To connect the Weil height of an algebraic number with the discriminant of its minimal polynomial, we shall use Mahler's inequality \cite{Mahler}. 

\begin{theorem}[Mahler]\label{mahler}
    Let $\alpha$ be a non-zero algebraic number of degree $d$ with minimal polynomial $m_\alpha(x)$ in $\Z[x]$. Then, 
    $$
        |D(m_{\alpha})| \leq d^d M(\alpha)^{2d-2}.
    $$
    
\end{theorem}

\medskip
\noindent
Therefore, using the definition of Weil height, we obtain

\begin{equation}
    h(\alpha) \geq \frac{1}{2} \left( \frac{\log|D(m_{\alpha})|}{d^2} - \frac{\log d}{d} \right).
\end{equation}

\bigskip

\section{\bf Proof of the main theorems}

\medskip

Let $K$ be a finite extension of $\Q_p$ with ramification index $e$ and inertia degree $f$. Let $\F_q$ denote the residue field of $K$, where $q = p^f$. We define $\nu$ to be the unique valuation on $K$ extending the usual $p$-adic valuation on $\Q_p$. Let $\alpha \in \overline{\Q} \setminus \mu_\infty$ be a non-zero algebraic number of degree $d$, and $\alpha_i$'s be the Galois conjugates of $\alpha$ in $\overline{\Q}$. Recall that
\begin{equation*}
\mathcal{S}_K(\alpha) = \left\{ \,\, \alpha_i\,\, | \,\,  \alpha_i \in  K,\,\, 1\leq i\leq d\right\}.
\end{equation*}

\medskip
\noindent
 Our proof is inspired by the proof of \cite[Theorem 2]{BZ} due to Bombieri and Zannier. 

% Throughout this section, we assume that for any extension $K$ of $\Q_p$, $\nu$ is the unique valuation extending the usual $p$ -adic valuation on $K$.
% and $\mathcal{O}_{\overline{\Q}}$ is the set of all algebraic integer. 

\begin{proof}[Proof of Theorem \ref{Main Theorem}]

 Let $m_\alpha(x)$ be the minimal polynomial of $\alpha$ in $\Z[x]$. Consider $L$ to be the splitting field of $m_\alpha(x)$ over $K$ and $\omega$ be the unique valuation on $L$ extending $\nu$ on $K_{\nu}$. Rearrange the conjugates $\alpha_i$ of $\alpha$ such that
 % Write
 %    \begin{equation*}
 %        m_\alpha(x) = a_d (x-\alpha_1) (x-\alpha_2)\cdots (x-\alpha_d), 
 %    \end{equation*}
 %    where $\alpha_i\in L_{\omega}$ satisfy
    \begin{equation*}
        \omega(\alpha_1)\geq \cdots \geq \omega(\alpha_r)\geq 0 > \omega(\alpha_{r+1})\geq \cdots \geq \omega(\alpha_d).
    \end{equation*}
The discriminant of $m_\alpha(x)$ is given by
\begin{align*}
    D(m_\alpha) & = a_d^{2d-2}\prod_{i<j} (\alpha_i - \alpha_j)^2.
\end{align*}
The contribution to the valuation of $D_{m_{\alpha}}$ from the terms in the product where either \( \omega(\alpha_j) < 0 \) or \( \omega(\alpha_i) < 0 \), can be bounded by

\begin{equation*}
    \omega\left(\prod_{j=r+1}^{d} \prod_{i=1}^{j-1} (\alpha_j - \alpha_i)\right) \geq \mathlarger{\mathlarger{\sum}}_{j=r+1}^d (j-1) \, \omega(\alpha_j).
\end{equation*}
Hence, we obtain
\begin{align}\label{eqn-1}
    \omega(D(m_\alpha)) & \geq (2d-2)\omega(a_d) + 2 \, \mathlarger{\mathlarger{\sum}}_{0<i<j\leq r} \omega(\alpha_j-\alpha_i) \, +  \, 2 \, \mathlarger{\mathlarger{\sum}}_{j=r+1}^{d} (j-1) \, \omega(\alpha_j) \nonumber\\
    & \geq 2 \, \mathlarger{\mathlarger{\sum}}_{i< j \leq r} \omega(\alpha_j - \alpha_i) \, - \,  2 \, \mathlarger{\mathlarger{\sum}}_{j=r+1}^{d} (d-j) \, \omega(\alpha_j).
\end{align}
Since $\omega(\alpha_j) < 0$ for $r+1 \leq j \leq d$, and $\omega(\alpha_1), \, \omega(\alpha_2), \ldots, \, \omega(\alpha_r)  \geq 0$, omitting a few non-negative terms if necessary, we have the following inequality
\begin{align}\label{eqn-11}
     \omega(D(m_\alpha)) & \geq 2 \, \mathlarger{\mathlarger{\sum}}_{i< j \leq r} \omega(\alpha_j - \alpha_i) \, - \, 2 \, \mathlarger{\mathlarger{\sum}}_{j=r+1}^{d} (d-j) \, \omega(\alpha_j) \nonumber\\ 
     &\geq 2 \, \mathlarger{\mathlarger{\sum}}_{\substack{i< j \leq r\\\alpha_i, \, \alpha_j \in \mathcal{S}_K(\alpha)}} \nu(\alpha_j - \alpha_i) \, - \,  2 \, \mathlarger{\mathlarger{\sum}}_{\substack{r<j\leq d \\ \alpha_j \in \mathcal{S}_K(\alpha)}} (d-j) \, \nu(\alpha_j).
\end{align}
For each $x \in \F_q$, let $N_{x}$ denote the number of conjugates of $\alpha_t \in \mathcal{S}_{K}(\alpha)$ with $1 \leq t \leq r$, which lie in the residue class $x \bmod \nu$. If $\alpha_i,\alpha_j \in \mathcal{S}_{K}(\alpha)$ lie in the same residue class modulo $\nu$, then $\nu(\alpha_i-\alpha_j) \geq \frac{1}{e}$. Hence, 
\begin{equation*}
    \mathlarger{\mathlarger{\sum}}_{\substack{i< j \leq r\\ \alpha_i, \alpha_j \in \mathcal{S}_K(\alpha)}} \nu(\alpha_j - \alpha_i) \geq \frac{1}{e} \mathlarger{\mathlarger{\sum}}_{x\in \F_q} \frac{N_{x} (N_{x} -1)}{2},
\end{equation*}
where $q=p^{f}$. 

\medskip

\noindent
Let $r'$ be the number of $\nu$-integral conjugates of $\alpha$ lying in $K$. Since there are $| \mathcal{S}_K(\alpha) |$ many conjugates of $\alpha$ lying in $K$, there are exactly $| \mathcal{S}_K(\alpha) | - r'$ non-integral conjugates of $\alpha$ in $K$. Therefore the second summation in \eqref{eqn-11} runs over all $| \mathcal{S}_K(\alpha) |-r'$ non-integral conjugates of $\alpha$ in $K$. As $d-j$ are all distinct positive integers, we have 
$$
\mathlarger{\mathlarger{\sum}}_{\substack{r<j\leq d \\ \alpha_j \in \mathcal{S}_K(\alpha)}} (d-j) \, \geq \, \frac{(| \mathcal{S}_K(\alpha) |-r') (| \mathcal{S}_K(\alpha) |-r'-1)}{2}.
$$
Note that, for $r < i  \leq d$, $ \nu(\alpha_i) \leq - \frac{1}{e}$  and therefore, \eqref{eqn-11} implies that

\begin{equation*}
   \operatorname{ord}_{p}(D(m_\alpha)) \geq \frac{1}{e} \left(    (| \mathcal{S}_K(\alpha) |-r')(| \mathcal{S}_K(\alpha) |-r'-1) + \mathlarger{\mathlarger{\sum}}_{x\in \F_q} N_x (N_x -1) \right),
\end{equation*}
where $\operatorname{ord}_p(D(m_\alpha))$ is the usual $p$-adic valuation of $D(m_\alpha)$. Since 
\begin{equation*}
(| \mathcal{S}_K(\alpha) |-r') + \sum_{x \in \F_q} N_x =| \mathcal{S}_K(\alpha) |, 
\end{equation*}
applying Cauchy-Schwarz inequality, we obtain
\begin{align*}
   \operatorname{ord}_{p}(D(m_\alpha))   &\geq \frac{1}{e} \left( (| \mathcal{S}_K(\alpha) |-r') (| \mathcal{S}_K(\alpha) |-r'-1) +  \sum_{x\in \F_q} N_x (N_x -1)  \right)\\
    & = \frac{1}{e} \left(   (| \mathcal{S}_K(\alpha) |-r')^2 + \sum_{x\in \F_q} N_x^2  + O(d) \right) \\
                & \geq  \frac{1}{e} \left(   \frac{| \mathcal{S}_K(\alpha) |^2}{q+1} + O(d) \right),
\end{align*}
where the implied constant has absolute value $\leq 1$. Therefore,
\begin{align*}
    \log |D(m_\alpha)| & \geq \frac{1}{e}| \mathcal{S}_K(\alpha) |^2 \, \left(\frac{\log p}{q+1}\right) + O\left( d \frac{\log p}{e}\right)\\
    & \geq  \frac{1}{[K : \Q_p]} | \mathcal{S}_K(\alpha) |^2 \, \left(\frac{\log q}{q+1}\right) + O\left( d \, \frac{\log q}{[K : \Q_p]}\right).
\end{align*}

% \noindent
%  Recall that the Mahler measure and the logarithmic Weil height of $\alpha$ is given by 
%  $$
% M(\alpha) = |a_n| \prod_{i=1}^{n}  \operatorname{max} \{ 1, |\alpha_i| \} \qquad \text{ and } \qquad h(\alpha) =  \frac{\log  M(\alpha)}{n}.
% $$
\noindent
Applying Mahler's inequality (Theorem \ref{mahler}), we deduce that
\begin{align}\label{lower bound on height}
    d \, h(\alpha) = \log M(\alpha) & \geq \frac{\log |D(m_\alpha)|}{2d} - \frac{\log d}{2} \nonumber \\ 
    & \geq \frac{1}{2 \, [K : \Q_p]} \left( \frac{| \mathcal{S}_K(\alpha) |^2}{d} \right) \frac{\log q}{(q+1)} + O\left( \frac{\log q} {[K : \Q_p]}\right) - \frac{\log d}{2},
\end{align}
where the implied constant in the $O$-term has absolute value $\leq 1/2$. Finally, using Northcott's Theorem \ref{Northcott} we conclude the proof of Theorem \ref{Main Theorem}.

\end{proof}

\medskip

\begin{proof}[Proof of Proposition \ref{Main Theorem 2}]
From \eqref{lower bound on height}, we deduce that if
 \begin{equation*}
      | \mathcal{S}_K(\alpha) | \geq  \sqrt{\left(\frac{2 (q+1) \, [K : \Q_p]\, d}{\log q} \right)\, \, \left( c + \frac{\log {(q d)}}{2} \right)  }
    \end{equation*}
then

\begin{equation*}
    d \, h(\alpha) \geq c + \frac{\log q}{2} + O \left( \frac{\log q}{[K: \Q_p]} \right).
\end{equation*}
Since the absolute value of the implied constant of the $O$-term is bounded above by $1/2$, we conclude that
\begin{equation*}
    h(\alpha) \geq \frac{c}{d}.
\end{equation*}
\end{proof}
\medskip

\begin{proof}[Proof of Corollary \ref{corollary I}] Taking $| \mathcal{S}_K(\alpha) | \geq (q^{3} \, [K : \Q_p]\, d\,\log d)^{1/2} $ in  \eqref{lower bound on height}, we obtain

\begin{equation}\label{equation-1}
 d \, h(\alpha) \,\,\,  \,\,\, \geq \,\,\, \left( \frac{q^3 \log q - q-1}{2(q+1)} \right) \log d + O\left( \frac{\log q} {[K : \Q_p]} \right).
 \end{equation}
Note that for any real number $x \geq 2$, 
\begin{equation*}
\frac{x^{3} \, \log x -   x-1 }{2(x+1)} \,\, \geq  \,\, \frac{2}{5}.
\end{equation*}
This is because the function in the LHS is increasing for $x\geq 2$ and the inequality holds for $x=2$. Applying this to \eqref{equation-1}, we obtain
\begin{equation*}
    d\, h(\alpha) \geq \,\,\, \frac{2}{5} \log d + O\left( \frac{\log q} {[K : \Q_p]} \right).
\end{equation*}
Therefore, as $d\to \infty$, the Mahler measure $M(\alpha) = d\,h(\alpha)$ also tends to infinity as required.
\end{proof}

\bigskip

\noindent
Now we are ready to prove Theorem \ref{Main Theorem 3}. Recall that for $\alpha \in \overline{\Q}$, 
$$
\rF_{p, \alpha} = \{ \, \Q_p(\alpha_i) \,\, | \,\, 1 \leq i \leq d \,\}.
$$
For each $K \in \mathcal{Q}_{p, \alpha}$, denote by $e_K$ and $f_K$ the ramification index and inertia degree of $K$, respectively, $n_K=[K:\Q_p]$ and $\F_{q_K}$ the residue field.  
\begin{proof}[Proof of Theorem \ref{Main Theorem 3}]
Let $\alpha\in \overline{\Q} \setminus \mu_\infty$ be an algebraic integer. For $K\in \rF_{p,\alpha}$ and $x\in \F_{q_K}$, define
\begin{equation*}
    N_{K, x} \, = \, \#\left\{\alpha_i \in \mathcal{S}_K(\alpha) \,\, : \,\, \alpha_i \equiv x \text{\, in \,} \F_{q_{K}} \right\}.
\end{equation*}
We can write

\begin{equation*}
|D(m_{\alpha})| = \prod_{p} {p}^{\operatorname{ord}_{p} (D(m_{\alpha}))}.
\end{equation*}
Now, it follows from the definition of $\mathcal{Q}_{p, \, \alpha}$ that
\begin{equation*}
    \operatorname{ord}_{p} (D(m_{\alpha})) =  \sum_p \sum_{ K \in \mathcal{Q}_{p,\alpha}} \nu_K(D(m_{\alpha})),
\end{equation*}
where $\nu_K$ is the unique valuation on $K$ extending the usual $p$-adic valuation on $\Q_p$.

\medskip

\noindent
Thus, we have
\begin{align}\label{disc-inequality}
\log |D(m_{\alpha})| & \geq  \sum_p \,\, \sum_{K \in \mathcal{Q}_{p, \alpha}} \left(\,\,\,\sum_{x \in \F_{q_K}} N_{K, x} \, ( N_{K, x} - 1) \, \,\right) \frac{\log p}{e_K} \nonumber \\ & \geq \sum_p  \log p\,\, \sum_{K\in \mathcal{Q}_{p, \alpha}} \frac{1}{e_K} \,\,\, \sum_{x \in \F_{q_K}} N_{K, x} \, ( N_{K, x} - 1).
\end{align}

\noindent
Since $ \sum_{x \in \F_{q_K}} N_{K, x} = | \mathcal{S}_K(\alpha) |$, using Cauchy-Schwarz inequality, we obtain

\begin{equation*}
 \sum_{K\in \mathcal{Q}_{p, \alpha}}\,\,  \frac{1}{e_K} \sum_{x \in \F_{q_K}} N_{K, x}^2  \, \, \geq \,\, \sum_{K\in \mathcal{Q}_{p, \alpha}}\,\,  \frac{1}{e_K \, q_K} | \mathcal{S}_K(\alpha) |^2.
\end{equation*}

\medskip
\noindent
As
$$
d = \sum_{K\in \mathcal{Q}_{p, \alpha}} | \mathcal{S}_K(\alpha) |, 
$$
\noindent
using Cauchy-Schwarz inequality, we get 
$$
\sum_{K\in \mathcal{Q}_{p, \alpha}} \,\,  \frac{1}{e_K q_K} | \mathcal{S}_K(\alpha) |^2 \,\,  \geq \,\, \frac{d^2}{\sum\limits_{K\in \mathcal{Q}_{p, \alpha}} e_K \, q_K}
$$

\noindent
and hence,
\begin{equation}\label{A}
     \sum_{K\in \mathcal{Q}_{p, \alpha}}\,\,  \frac{1}{e_K} \sum_{x \in \F_{q_K}} N_{K, x}^2  \, \, \geq \,\, \frac{d^2}{\sum\limits_{K\in \mathcal{Q}_{p, \alpha}} e_K \, q_K}. 
\end{equation}

\medskip
\noindent
On the other hand
\begin{equation}\label{B}
      \sum_{K\in \mathcal{Q}_{p, \alpha}}\,\,  \frac{1}{e_K} \sum_{x \in \F_{q_K}} N_{K, x}  \, \, = \,\,  \sum\limits_{K\in \mathcal{Q}_{p, \alpha}}\,\, \frac{| \mathcal{S}_K(\alpha) |}{e_K}. 
\end{equation}

\medskip

\noindent
Combining \eqref{A} and \eqref{B} in \eqref{disc-inequality}, we get
\begin{align*}
\log |D(m_\alpha)| \, \, \geq \, \, d^2 \, \, \mathlarger{\mathlarger{\sum}}_p \,\, \left( \frac{1}{\sum\limits_{K \in \mathcal{Q}_{p, \alpha}} e_K \, q_K}  - \frac{1}{d^2} \sum_{K\in \mathcal{Q}_{p, \alpha}}\,\, \frac{| \mathcal{S}_K(\alpha) |}{e_K} \right) \log p.
\end{align*}
\noindent
Finally, applying Mahler's inequality (Theorem \ref{mahler}), we conclude the theorem.
\end{proof}

\medskip

\begin{proof}[Proof of Corollary \ref{corollary 1}]
By \eqref{disc-inequality}, we have the inequality 
\begin{equation*}
\log |D(m_{\alpha})| \geq \sum_p \,\, \sum_{K\in \mathcal{Q}_{p, \alpha}} \,\,\, \sum_{x \in \F_{q_K}} N_{K, x} \, ( N_{K, x} - 1) \, \, \frac{\log q_K}{n_K}.
\end{equation*}

\noindent
For $K \in \mathcal{Q}_{p, \alpha}$, since $\sum\limits_{x \in \mathbb{F}_{q_K}} N_{K, x} = | \mathcal{S}_K(\alpha) | $ is the total number of conjugates of $\alpha \in K$, using Cauchy-Schwarz inequality, we obtain
\begin{align*}
  \log |D(m_{\alpha})| \geq  \sum_{K \in \mathcal{Q}_{p, \alpha}}\,\,\, \sum_{x \in \F_{q_K}} N_{K, x} \, ( N_{K, x} - 1) \, \, \frac{\log q_K}{n_K} &=  \sum_{K \in \mathcal{Q}_{p, \alpha}} \, \frac{\log q_K}{n_K} \, \sum_{x \in \mathbb{F}_{q_K}}  N_{K, x} \, ( N_{K, x} - 1) \\
    &\geq \sum_{ \substack{ K \in \mathcal{Q}_{p, \alpha}}} \left( \frac{| \mathcal{S}_K(\alpha) |^2}{q_K} - | \mathcal{S}_K(\alpha) | \right) \, \, \frac{\log q_K}{n_K}.
\end{align*}
The proof now follows by applying Mahler's inequality (Theorem \ref{mahler}).
\end{proof}

\medskip

% \noindent \textbf{Data Availability Statement:} This manuscript has no associated data.

% \medskip

% \section{\bf Acknowledgments}
% We thank Prof. Sinnou David for several fruitful discussions. We also thank Dr. Siddhi Pathak for helpful comments on an earlier version of this paper.
% \medskip

\end{document}